\documentclass{article}

\usepackage[preprint]{neurips_2021}

\usepackage[utf8]{inputenc} % allow utf-8 input
\usepackage[T1]{fontenc}    % use 8-bit T1 fonts
\usepackage{hyperref}       % hyperlinks
\usepackage{url}            % simple URL typesetting
\usepackage{booktabs}       % professional-quality tables
\usepackage{amsfonts}       % blackboard math symbols
\usepackage{nicefrac}       % compact symbols for 1/2, etc.
\usepackage{microtype}      % microtypography
\usepackage{xcolor}         % colors

\usepackage{amsthm}
\usepackage{amsmath}
\usepackage{amssymb}
\newcommand{\bedit}[1]{{\color{black}#1}}

\newtheorem{proposition}{Proposition}
\newtheorem{theorem}{Theorem}

\theoremstyle{definition}
\newtheorem{definition}{Definition}
\newtheorem{remark}{Remark}

\DeclareMathOperator{\vspan}{span}

\title{Occupation Kernel Hilbert Spaces for Fractional Order Liouville Operators and Dynamic Mode Decomposition}

\author{%
  Joel A. Rosenfeld \\
  Department of Mathematics and Statistics\\
  University of South Florida\\
  Tampa, Fl 33620 \\
  \texttt{rosenfeldj@usf.edu} \\
  \And
  Benjamin P. Russo \\
  Department of Mathematics\\
  Farmingdale State College\\
  Farmingdale, NY 11735 USA \\
  \texttt{russobp@farmingdale.edu} \\
  \And
  Xiuying Li \\
  Changshu Institute of Technology\\
  Suzhou, Jiangsu 215500, P R China \\
  \texttt{xyli1112@sina.com} \\
  }

\begin{document}

\maketitle

\begin{abstract}
  This manuscript gives a theoretical framework for a new Hilbert space of functions, the so called occupation kernel Hilbert space (OKHS), that operate on collections of signals rather than real or complex numbers. To support this new definition, an explicit class of OKHSs is given through the consideration of a reproducing kernel Hilbert space (RKHS). This space enables the definition of nonlocal operators, such as fractional order Liouville operators, as well as spectral decomposition methods for corresponding fractional order dynamical systems. In this manuscript, a fractional order DMD routine is presented, and the details of the finite rank representations are given. Significantly, despite the added theoretical content through the OKHS formulation, the resultant computations only differ slightly from that of occupation kernel DMD methods for integer order systems posed over RKHSs.
\end{abstract}

\section{Introduction}\label{sec:introduction}

Despite the proliferation of numerical methods for and applications of fractional order and nonlocal dynamical systems over the past twenty years (cf. \citep{rosenfeld2017approximating,diethelm2010analysis,pang2019fpinns,magin2006fractional}), there are several long standing problems in the modeling of nonlocal dynamical systems and system identification techniques for nonlinear fractional order systems (cf. \citep{Das2012,rahmani2018nonlinear,dorcak2013identification}). Principle among these problems is the effective representation of data for modeling in nonlocal nonlinear systems. This manuscript enables several novel data driven approaches to modeling nonlocal dynamical systems, including a nonlocal variant of Dynamic Mode Decomposition (DMD), by addressing the problem of data representation for nonlinear time-fractional order dynamical systems. In particular, this work builds on work involving occupation kernels, which embeds signal information into a function within a reproducing kernel Hilbert space (RKHS) and thus positions signal data as the fundamental unit of information of a dynamical system (cf. \citep{rosenfeld2019occupation,rosenfeldCDC2019,rosenfeldDMD}), by extending the idea beyond merely functions inside of a space, but by generating Hilbert spaces of functions on collections of signals based on the occupation kernels themselves.

Fractional order dynamical systems have been applied broadly in nearly every scientific discipline. For example, applications include two-phase flows \citep{xie2017direct,zheng2016phase}, turbulence modeling \citep{song2018universal}, and stochastic systems (e.g. \citep{du2019stochastic}). Other research has exploited the memory capabilities of fractional order operators in the context of materials, where classically they were used to model visco-elastic materials \citep{mainardi2010fractional} and more recently biological tissues \citep{magin2012fractional}, and fractional order methods have also been employed in biological applications for guidance in sensor placement \citep{tzoumas2018selecting}. Extensions of control theory have been realized through fractional order PID controllers \citep{podlubny1994fractional}, which have seen applications to problems in aerospace and control of flexible systems \citep{muresan2018novel,sirota2015fractional}.

Despite the wide application of fractional order systems, given a real world system, the development of a corresponding fractional order model is nontrivial. In the absence of first principles through which a model may be realized, many applications of fractional order dynamical systems leverage fractional operators in an unmotivated way. This manuscript provides a new Hilbert space setting for the modeling of nonlinear fractional order systems, which will provide a data driven motivation for the implementation of fractional order operators and models.

System identification and learning for fractional order systems has been largely constrained to linear fractional order systems. For example, \citep{rahmani2018nonlinear} developed a system identification routine through the Laplace transform, while \citep{gulian2019machine} gives a learning method for a collection \emph{linear} fractional order PDEs. A notable exception is fPiNNs (cf.  \citep{mehta2019discovering} and \citep{pang2019fpinns}), which uses physics informed neural network to provide a model of a fractional order PDE, and also that of the authors which leveraged a regression approach with occupation kernels \citep{li2020fractional}. The methods herein differ from that of fPiNNs in that they are operator based (rather than neural network based) approaches to system identification and modeling, and constitute a contribution in a new direction for both data driven modeling for fractional order system and for DMD analysis techniques as a whole.

One of the hurtles that have prevented the development of system identification for nonlocal nonlinear systems is a matter of data representation, which is resolved in the sequel by the invocation of occupation kernels. This builds on the work in \citep{li2020fractional} in that some of the occupation kernels take the same form as in \citep{li2020fractional}. The objective of the work of \citep{li2020fractional} was to obtain an approximation of the dynamics themselves. The principle difference is the development of nonlocal operators and a new Hilbert space framework, which allows for spectral decompositions of these operators for expressing a model for the system.

Section \ref{sec:okhs} introduces occupation kernel Hilbert spaces (OKHSs). The construction presented here is for the particular case of the Caputo fractional derivative, but is immediately generalizable to a wide range of time fractional operators where initial value problems may be resolved using Voltera integral equations. The development of OKHSs enables the generalization of a key operator used in the study of nonlinear dynamical systems, namely the Liouville operator. Some Liouville operators are continuous generators of semigroups of Koopman operators \citep{cvitanovic2005chaos,das2020koopman,froyland2014detecting,giannakis2019data,giannakis2020extraction,giannakis2018koopman}, which are the pivotal tools in the study of nonlinear dynamical systems of integer order.

\subsection{Overview of Dynamic Mode Decomposition Methods}

DMD emerged as an effective data-driven method of learning dynamical systems from trajectory data with no prior knowledge. The DMD method aims to analyze finite dimensional nonlinear dynamical systems as operators over infinite dimensional Hilbert spaces, where the tools of traditional linear systems may be employed to make predictions about nonlinear systems from captured trajectory data (or snapshots). Though based on early work by Koopman and Von Neumann in the 1930s, DMD more recently came to prominence as a method of identifying underlying governing principles of nonlinear fluid flows (cf. \citep{budivsic2012applied,kutz2016dynamic,mezic2005spectral,mezic2013analysis,williams2015data}), which compared favorably with principle orthogonal decomposition (POD) analyses.

Underlying classical DMD methods are Koopman operators, which are operators over function spaces that represent discrete time dynamics \citep{kutz2016dynamic}. That is, given a discrete dynamical system, $x_{i+1} = F(x_i)$, and a Hilbert space of functions, $H$, the corresponding Koopman operator, $\mathcal{K}_F : \mathcal{D}(\mathcal{K}_F) \to H$ is given as $\mathcal{K}_F g = g\circ F$ for all $g \in \mathcal{D}(\mathcal{K}_F) \subset H$, where $\mathcal{D}(\mathcal{K}_F)$ is a given subset of $H$ corresponding to the domain of $\mathcal{K}_F$. For a collection of snapshots of a dynamical system, $x_1, x_2,\ldots, x_m$, the dynamics are then represented through observables pulled from $H$, as $g(x_1), g(x_2),\ldots, g(x_m)$, where $g(x_{i+1}) = g(F(x_i)) = \mathcal{K}_F g(x_i)$ \citep{kutz2016dynamic,williams2015data}. When $H$ is a reproducing kernel Hilbert space (RKHS) over $\mathbb{R}^n$ with kernel function $K:\mathbb{R}^n \times \mathbb{R}^n \to \mathbb{R}$ (the definition of which is given in the Technical Background below), then this action on the observables may be expressed as \[ \langle g, K(\cdot,x_{i+1}) \rangle_H = g(x_{i+1}) = \mathcal{K}_Fg(x_i) = \langle \mathcal{K}_F g, K(\cdot, x_i) \rangle_H = \langle g, \mathcal{K}_F^* K(\cdot,x_i) \rangle_H. \]

Hence, discrete time dynamics may be captured through the adjoint of the Koopman operator acting on the kernel functions of a RKHS.\footnote{In the context of function theoretic operator theory, the function $F$, which defines the discrete time dynamics, is called the \emph{symbol} of the Koopman operator. This terminology is common among many other such operators over RKHSs (cf. \citep{rosenfeld2016sarason,rosenfeld2015introducing,sarason2008unboundedtoeplitz})} This connection between Koopman operators and RKHSs is the core observation of kernel-based extended DMD presented in \citep{williams2015kernel, korda2018convergence}. %The kernel-based extended DMD algorithm then expands the kernel function as a column of features, collects a sequence of input-output data, and then reduces the system through a singular value decomposition. The DMD procedure is completed upon determining the eigenvectors and eigenvalues of the reduced system, and then lifting up these eigenvalues and eigenfunctions to the Koopman operator to establish what are called Koopman modes for the dynamical system \citep{williams2015kernel}.

%So described is the DMD procedure for discrete time dynamical systems.  Practitioners analyze continuous time dynamics by first fixing a time-step and then examining a discrete time proxy for the continuous time system \citep{kutz2016dynamic}. The intuition employed in this approach is that by taking small time steps, an adequate approximation of the Liouville operator may be determined by using a Koopman operator.

The limiting perspective using Koopman operators has several major theoretical drawbacks; the most significant and fundamental is its restriction to dynamical systems that admit a discretization. The symbol (or discrete dynamics) of a Koopman operator must be defined over all of the state space for the Koopman operator to be well defined. The consideration of the dynamics $\dot x = 1 + x^2$ yields a discretization of $x_{m+1} = \tan(\Delta t + \arctan(x_{m}))$, where it can be seen that the selection of $x_m = \tan(\pi/2 - \Delta t)$ gives an undefined value for $x_{m+1}$. Hence, many polynomial systems do not fall under the purview of Koopman based techniques, and the particular requirement that guarantees the existence of a discretization of a dynamical system is \emph{forward completeness} (cf. \citep[Chapter 11]{krstic2009delay}) which is usually established by demonstrating that a dynamical system is globally Lipschitz continuous (cf. \citep{coddington1955theory} and \citep[Theorem 3.2]{khalil2002nonlinear}).

To address the above limitation of the Koopman perspective, \citep{rosenfeldDMD} introduced a combination of Liouville operators and occupation kernels to determine a finite rank representation of Liouville operators for spectral analysis. This allows for the analysis of dynamical systems that do not admit discretizations, and as such, this allows for the analysis of Liouville operators that were not also Koopman generators. Moreover, the introduction of occupation kernels and scaled Liouville operators in \citep{rosenfeldDMD} over the Bargmann-Fock space yielded a compact operator for a wide range of dynamics, and scaled Liouville operators allows for norm convergence of finite rank approximations used in DMD procedures. The main thrust of this manuscript is to provide a framework for the transportation of these tools to (time) fractional order dynamical systems, where data driven modeling techniques using DMD procedures could not be previously applied.

\section{Motivation for the present approach}

Suppose that for $0 < q < 1$ the Caputo fractional derivative is given as\footnote{Here $C_q := 1/\Gamma(q)$ is used to avoid confusion with the occupation kernels.} $D_*^q x(t) = C_{1-q} \int_0^t (t-\tau)^{-q} \dot x(\tau)\, d\tau$, and let $D_*^q x(t) = f(x(t))$ be a nonlinear fractional order system with the initial condition $x(0) = x_0 \in \mathbb{R}^n$. The Riemann-Liouville fractional integral is given as $J^q x(t) = C_{q} \int_0^t (t-\tau)^{q-1} x(\tau)\, d\tau,$ and the initial value problem may be resolved as $x(t) = x_0 + J^{q} D_*^q x(t) = x_0 + J^q f(x(t))$ (cf. \citep{diethelm2010analysis,gorenflo1997fractional}). Hence, $\dot x(t) = \frac{d}{dt} J^q f(x(t)) = D^{1-q} f(x(t))$, where $D^q$ is the Riemann-Liouville fractional derivative operator (cf. \citep{diethelm2010analysis,gorenflo1997fractional,oldham1974fractional}). Exploiting this core idea, generalizations of Liouville operators and occupation kernels may be defined.

Specifically, the variant of the Liouville operator explored in this manuscript is the \emph{fractional} Liouville operator of order $q$ given as $\mathcal{A}_{f,q} g = \bedit{\mathcal{A}}_{f,q} g[\gamma](T_\gamma) = \frac{1}{\Gamma(q)} \int_0^{T_\gamma} (T_\gamma - t)^{-q} \nabla g(\gamma(t)) D^{1-q} f(\gamma(t)) dt$  where $g$ is a function on a collection of signals in a Hilbert space that will be defined in Section \ref{sec:okhs}. If the fractional Liouville operator was poised over a RKHS consisting of functions over $\mathbb{R}^n$, then the image of $\mathcal{A}_{f,q}g$ would be a function of $\mathbb{R}^n$, which would mean that for any $x \in \mathbb{R}^n$, the quantity $\mathcal{A}_{f,q} g(x)$ would have to be well defined. However, as a state carries no ``history,'' the term $D^{q-1} f(x)$ is not well defined. This difficulty prevents the straightforward generalization of Liouville operators for integer order dynamical systems to that of fractional order dynamical systems. Section \ref{sec:okhs} of this manuscript provides a resolution to this problem through the construction of a Hilbert space consisting of functions over signals. Thus, in the sequel, the domain of $\mathcal{A}_{f,q} g$ is a collection of signals, and for a given continuously differentiable signal $\theta: [0,T] \to \mathbb{R}^n$, the quantity $D^{q-1} f(\theta(t))$ is well defined. The astute reader will ask how a gradient of a functions, $g$, over a collection of signal is defined and this will be quantified more precisely in Section \ref{sec:okhs}.

To accommodate the nonlocal requirements of the fractional Liouville operator, Section \ref{sec:okhs} constructs a Hilbert space from a RKHS, $H_{RKHS}$. Specifically, for any function, $g$, in a continuously differentiable RKHS over $R^n$ a mapping over collection of signals arises naturally via $\phi_g[\theta](t) := g(\theta(t))$. Hence, $\phi_g$ maps the signal $\theta$ to the signal $g(\theta(\cdot))$. In combination with the canonical identification, $g \mapsto \phi_g$, and the inner product of the RKHS, $\langle \cdot, \cdot \rangle_{H_{RKHS}}$, an inner product on the vector space  $H_{OKHS}^\circ := \{ \phi_g : g \in H_{RKHS}\}$ may be established, and the resultant Hilbert space will be called an \textit{occupation kernel Hilbert space} (OKHS).  The name for this Hilbert space arises from occupation kernels, which will play a key role in the analysis of nonlinear dynamical systems of fractional order, just as they have for integer order systems (cf. \citep{rosenfeld2019occupation}). For a given signal, $\theta: [0,T] \to \mathbb{R}^n$, the occupation kernel of order $q > 0$ is given as $\Gamma_{\theta,q}(x) := C_q \int_0^T (T-\tau)^{q-1} K(x,\theta(\tau))\, d\tau$. Occupation kernels represent integration after composition of the signal with a function in a RKHS, $\langle g, \Gamma_{\theta,q} \rangle = C_q \int_0^T (T-\tau)^{q-1} g(\theta(\tau))\, d\tau$.

The introduction of OKHSs and associated Liouville operators directly addresses the problem of data integration. That is, OKHSs directly incorporate trajectory data in a kernel function contained in the Hilbert space. The definition of the Hilbert space, which consists of functions on \emph{signals}, allows for the definition of nonlocal Liouville operators. In turn, these nonlocal Liouville operators allow for the development of DMD procedures on nonlinear fractional order systems.

\section{Prerequisites}

\subsection{Reproducing Kernel Hilbert Spaces}
A (real-valued) reproducing kernel Hilbert space (RKHS) is a Hilbert space, $H$, of real valued functions over a set $X$ such that for all $x \in X$ the evaluation functional $E_x g := g(x)$ is bounded (cf. \citep{paulsen2016RKHS,steinwart2008support,wahba1999support,aronszajn1950theory,bishop2006pattern,rosenfeld2015densely,rosenfeld2018mittag}). As such, the Riesz representation theorem guarantees, for all $x \in X$, the existence of a function $k_x \in H$ such that $\langle g, k_x \rangle_H = g(x)$, where $\langle \cdot, \cdot \rangle_H$ is the inner product for $H$. The function $k_x$ is called the reproducing kernel function at $x$, and the function $K(x,y) = \langle k_y, k_x \rangle_H$ is called the kernel function corresponding to $H$.

To establish a connection between RKHSs and nonlinear dynamical systems the following operator was introduced, which was inspired by the study of occupation measures and related concepts (cf. \citep{lasserre2008nonlinear,bhatt1996occupation,majumdar2014convex,henrion2008nonlinear,savorgnan2009discrete,hernandez2000classification,budivsic2012geometry,budivsic2012applied,cvitanovic1995dynamical,levnajic2010ergodic,levnajic2015ergodic,mezic2005spectral,mezic1999method,susuki2020invariant}).

\begin{definition}Let $\dot x = f(x)$ be a dynamical system with the dynamics, $f: \mathbb{R}^n \to \mathbb{R}^n$, Lipschitz continuous, and suppose that $H$ is a RKHS over a set $X$, where $X \subset \mathbb{R}^n$ is compact. The \emph{Liouville operator with symbol $f$}, $A_f : \mathcal{D}(A_f) \to H$, is given as $A_f g := \nabla_x g \cdot f,$ where $\mathcal{D}(A_f):= \left\{ g \in H : \nabla_x g \cdot f \in H \right\}.$ \end{definition}

As a differential operator, $A_f$ is not expected to be a bounded over many RKHSs. However, as differentiation is a closed operator over RKHSs consisting of continuously differentiable functions (cf. \citep{steinwart2008support,reed2012methods}), it can be similarly established that $A_f$ is closed under the similar circumstances (cf. \citep{rosenfeld2019occupation}). Thus, $A_f$ is a closed operator for RKHSs consisting of continuously differentiable functions. Consequently, the adjoints of densely defined Liouville operators are themselves densely defined (cf. \citep{pedersen2012analysis}).

Associated with Liouville operators in particular are a special class of functions within the domain of the Liouville operators' adjoints. The following definition and proposition were given in \citep{rosenfeld2019occupation}, and these generalize the framework of \citep{lasserre2008nonlinear} on occupation measures.
\begin{definition}\label{def:occ}Let $X \subset \mathbb{R}^n$ be compact, $H$ be a RKHS of continuous functions over $X$, and $\gamma:[0,T] \to X$ be a continuous signal. The functional $g \mapsto \int_0^T g(\gamma(\tau)) d\tau$ is bounded, and may be respresented as $\int_0^T g(\gamma(\tau)) d\tau = \langle g, \Gamma_{\gamma}\rangle_H,$ for some $\Gamma_{\gamma} \in H$ by the Riesz representation theorem. The function $\Gamma_{\gamma}$ is called the occupation kernel corresponding to $\gamma$ in $H$.\end{definition}

\begin{proposition}\label{prop:occupation}
Let $H$ be a RKHS of continuously differentiable functions over a compact set $X$, and suppose that $f:\mathbb{R}^n \to \mathbb{R}^n$ is Lipschitz continuous. If $\gamma:[0,T] \to X$ is a trajectory %as in Definition \ref{def:occ}
that satisfies $\dot \gamma = f(\gamma)$, then $\Gamma_{\gamma} \in \mathcal{D}(A_f^*)$, and $A_f^* \Gamma_{\gamma} = K(\cdot,\gamma(T)) - K(\cdot,\gamma(0)).$
\end{proposition}

Proposition \ref{prop:occupation} thus integrates nonlinear dynamical systems with RKHSs through the Liouville operator. In particular, the action of the adjoint of the Liouville operator acting on an occupation kernel expressed as a difference of kernels, enables continuous time DMD analyses that do not require an a priori discretization of the dynamical system \citep{rosenfeldDMD}.

\subsection{Time Fractional Dynamical Systems}

Time fractional derivatives, such as the Riemann-Liouville fractional derivative and the Caputo fractional derivative, are defined with respect to two operators: integer order derivatives and the Riemann-Liouville fractional integral given as $J^q \gamma(t) := \frac{1}{\Gamma(q)} \int_0^T (t-\tau)^{q-1} \gamma(\tau) d\tau.$ The Riemann-Liouville fractional derivative of order $0 < q < 1$ of a function $\gamma$ is then given as $\frac{d}{dt} J^{1-q} \gamma(t)$, while the Caputo fractional derivative of order $0 < q < 1$ is given as $D_*^q \gamma(t) := J^{1-q} \gamma'(t) = \frac{1}{\Gamma(1-q)} \int_0^t (t-\tau)^{-q} \gamma'(\tau) d\tau$. Initial value problems for the Caputo fractional derivative may be expressed as \begin{equation}\label{eq:caputo-dynamics}D_*^q x = f(x) \text{ with } x(0) = x_0,\end{equation} and their solution can be written in terms of a Volterra operator as $x(t) = x_0 + \frac{1}{\Gamma(q)} \int_0^t (t-\tau)^{q-1} f(\tau) d\tau.$ Riemann-Liouville based initial value problems require initial values with respect to fractional derivatives, which are difficult to express for a physical system \citep{diethelm2010analysis}. Hence, the present manuscript will focus on dynamical systems arising from the Caputo fractional derivative.

\section{Occupation Kernel Hilbert Spaces}\label{sec:okhs}

To provide for a nonlocal Hilbert space framework, the objective of Section \ref{sec:okhs} is to develop further the concept of occupation kernels. At this point, occupation kernels have been viewed as a part of a RKHS rather than as a generator of a Hilbert space on their own. In what follows, this will remain true to some extent, where many OKHSs will arise from RKHSs. However, the central objects of study will be bounded functionals that come from trajectories. The focus on trajectories rather than points in $\mathbb{R}^n$ allows for the treatment of operators on nonlocal information. Section \ref{sec:okhs} formalizes the necessary Hilbert space framework, which will enable the development of DMD analyses of (time) fractional order dynamical systems.

\begin{definition}\label{def:okhs}Let \bedit{$q>0$}. For fixed $d,e \in \mathbb{N}$ with $e < d$, let $\mathcal{X} = \cup_{T > 0} C^d([0,T],\mathbb{R}^n)$ and $\mathcal{Y} = \cup_{T > 0} C^e([0,T],\bedit{\mathbb{R}})$. The OKHS of order $q$, $H$, over $\mathcal{X}$ is a Hilbert space that consists of functions $g : \mathcal{X} \to \mathcal{Y}$ such that \bedit{for each $\gamma\in \mathcal{X}$ and $g[\gamma]: [0,T] \to \mathbb{R}$ in $\mathcal{Y}$ the functional $g \mapsto J^{q}g[\gamma] = C_q \int_0^{T} (T-t)^{q-1} g[\gamma](t) dt$} is bounded, with $g[\gamma] \in Y$ indicating the mapping of $\gamma$ from $\mathcal{X}$ to $\mathcal{Y}$ by $g \in H$.\end{definition}
For each such functional, the Riesz representation theorem yields a function, $\Gamma_{\gamma,q} \in H$, such that $\langle g, \Gamma_{\gamma,q} \rangle_H = \frac{1}{\Gamma(q)} \int_0^{T_{\gamma}} (T_{\gamma}-t)^{q-1} g[\gamma](t) dt,$ this function is called the occupation kernel corresponding to $H$.

\subsection{OKHSs from RKHSs}

The existence of OKHSs follows immediately from that of RKHSs. 
\bedit{
\begin{definition}
Let $\mathcal{X} = \cup_{T > 0} C^d([0,T],\mathbb{R}^n)$ and $\mathcal{Y} = \cup_{T > 0} C^e([0,T],\mathbb{R}^m)$ and $\mathcal{V}$ be a vector space of infinitely differentiable functions $g:\mathbb{R}^n\rightarrow\mathbb{R}^m$. For each $g\in \mathcal{V}$ let $\phi_g : \mathcal{X} \to \mathcal{Y}$ be the mapping that takes $\gamma \in \mathcal{X}$ to $\phi_g[\gamma] := g(\gamma(\cdot))\in \mathcal{Y}$. Let $\mathcal{W}$ be the vector space of maps $\varphi:\mathcal{X}\rightarrow \mathcal{Y}$, under the standard operations. The operator, $\mathcal{T}:\mathcal{V}\rightarrow \mathcal{W}$ given by $\mathcal{T} g := \phi_g$, is called the canonical mapping of the scalar valued function $g$ to the signal valued function $\phi_g$.
\end{definition}
Note, since $g$ is infinitely differentiable, $\phi_g[\gamma]$ is at least as smooth as $\gamma$. Hence $\phi_g[\gamma] \in Y$.
}

% Let $H_{RKHS}$ be a RKHS of infinitely differentiable functions, such as the native space for the Gaussian RBF, then for each $g \in H_{RKHS}$ there is an associated mapping from $X$ to $Y$ denoted by $\phi_g : X \to Y$ that takes $\gamma \in X$ and maps it as $\phi_g[\gamma] := g(\gamma(\cdot)).$ Since $g$ is infinitely differentiable, $\phi_g[\gamma]$ is at least as smooth as $\gamma$. Hence $\phi_g[\gamma] \in Y$. The map $\mathcal{T} g := \phi_g$ is called the canonical mapping of the scalar valued function $g$ to the signal valued function $\phi_g$.

\begin{theorem}
Let $H_{RKHS}$ be an RKHS of infinitely differentiable \bedit{scalar-valued} functions over $\mathbb{R}^n$ with kernel function $K$. The space, $\mathcal{T}H_{RKHS} := \{ \mathcal{T}g : g\in H_{RKHS}\}$ is an OKHS of any order $q > 0$ where the inner product is taken to be $\langle \phi, \psi \rangle_{OKHS} = \langle \mathcal{T}^{-1} \phi, \mathcal{T}^{-1} \psi \rangle_{RKHS}.$
\end{theorem}

\begin{proof}
From the discussion before the theorem statement, $\phi_g$ is well defined as a map from $X$ to $Y$ for any $g \in H_{RKHS}$. To demonstrate that $H_{OKHS} := \mathcal{T} H_{RKHS}$ is an occupation kernel Hilbert space, it must be demonstrated that the space is complete with respect to the norm induced by the inner product and that for any $q > 0$, the collection of functionals in Definition \ref{def:okhs} are bounded.

Note that $\mathcal{T}$ is linear, and hence $H_{OKHS}$ is a vector space. That $\mathcal{T}$ is injective follows from the consideration of constant signals. If $g_1,g_2 \in H_{RKHS}$ are distinct functions, then there is a point $x \in \mathbb{R}^n$ such that $g_1(x) \neq g_2(x)$. Considering $\theta_x(t) \equiv x$ for $t\in [0,1]$ which resides in $X$ given in Definition \ref{def:okhs}, it follows that $\phi_{g_1}[\theta_x](t) = g_1(\theta_x(t)) = g_1(x) \neq g_2(x) = \phi_{g_2}[\theta_x](t)$ for all $t$.

Once injectivity is established, the inner product on $H_{OKHS}$, given as \[\langle \phi, \psi \rangle_{OKHS} := \langle \mathcal{T}^{-1} \phi, \mathcal{T}^{-1} \psi \rangle_{RKHS},\] is well defined. Moreover, $\mathcal{T}$ is automatically continuous with respect to the induced norm on $H_{OKHS}$.

To demonstrate completeness, suppose that $\{ \phi_m \}_{m=1}^\infty \subset \mathcal{T} H_{RKHS}$ is Cauchy with respect to the norm induced by the inner product on $H_{OKHS}$. For each $m$, there is a function $g_m \in H_{RKHS}$ such that $\phi_{m} = \mathcal{T} g_m$. Given any $\epsilon > 0$, there is an $M$ such that for all $m,m' > M$, the following holds,
\begin{gather*}
\| \phi_m - \phi_{m'} \|_{OKHS}^2 = \sqrt{ \langle \phi_m - \phi_{m'}, \phi_m - \phi_{m'} \rangle_{OKHS}} \\
\sqrt{ \langle g_m - g_{m'}, g_m - g_{m'} \rangle_{RKHS}} = \| g_m - g_{m'} \|^2_{RKHS}.
\end{gather*}
Hence, $g_m$ is a Cauchy sequence in $H_{RKHS}$, and there is a function $g \in H_{RKHS}$ such that $g_m \to g$. The continuity of the canonical identification yields $\phi_{g_m} \to \phi_g$, and thus the limit of $\phi_m$ resides in $\mathcal{T} H_{RKHS}$. Therefore, $\mathcal{T} H_{RKHS}$ is complete and $\mathcal{T} H_{RKHS} = H_{OKHS}.$

To demonstrate the boundedness of the functional $\phi \mapsto C_q \int_0^{T} (T-\tau)^{q-1} \phi[\gamma](\tau) d\tau$ for each $\gamma \in X$ and $q > 0$, let $g = \mathcal{T}^{-1}\phi$. It follows that
\begin{gather*}
    \left| C_q \int_0^{T} (T-\tau)^{q-1} \phi[\gamma](\tau)d\tau \right| = \left| C_q \int_0^{T} (T-\tau)^{q-1} g(\gamma(\tau)) d\tau\right| \\
    \le \| g\|_{RKHS} \cdot C_q \int_0^{T} (T-\tau)^{q-1} \sqrt{K(\gamma(\tau),\gamma(\tau))} d\tau\\
    =\| \phi \|_{OKHS} \cdot C_q \int_0^{T} (T-\tau)^{q-1} \sqrt{K(\gamma(\tau),\gamma(\tau))} d\tau.
\end{gather*}
As $K$ is a continuous function on $\mathbb{R}^n$ and $\gamma([0,T])$ is compact in $\mathbb{R}^n$, it follows that the last integral is bounded for any \bedit{$q>0$}. Note that for $0 < q < 1$, the quantity $(T-\tau)^{q-1}$ has an integrable singularity at $\tau = T$.
\end{proof}
\bedit{
\begin{remark}
The operator $\mathcal{T}:H_{RKHS}\rightarrow H_{OKHS}$ is an isometric isomorphism from the base reproducing kernel Hilbert space to the constructed occupation kernel Hilbert space. %In fact, 
% \[\|\mathcal{T}\|=\sup_{g\neq 0}\left\{\frac{\|\mathcal{T}g\|}{\|g\|}\ :\ g\in H_{RKHS}\right\}=\sup_{g\neq 0}\left\{\frac{\|g\|}{\|g\|}\ :\ g\in H_{RKHS}\right\}=1.\]

\end{remark}
\begin{remark}\label{sametimeframe}
In general, if $g$ is arbitrary signal valued function, then $g[\gamma]$ and $\gamma$ are not necessarily defined over the same interval. However, if $H$ is an OKHS arising from an RKHS under the canonical mapping, then given a $g\in H$ and $\gamma:[0,T]\rightarrow\mathbb{R}^n$ it follows that $g[\gamma]$ is also defined over $[0,T]$.
\end{remark}

}

% \subsection{RKHSs from OKHSs}

% While RKHSs naturally provide examples of OKHSs, the converse is also true.

% {\color{black}
% \begin{theorem}
% Let $H_{OKHS}$ be an OKHS of order $q > 0$, then for each $x \in \mathbb{R}^n$ function evaluation can be defined through the consideration of the constant signal $\theta_x : [0,1] \to \mathbb{R}^n$ defined as $\theta_x \equiv x$.  In particular evaluation of $\phi \in H_{OKHS}$ may be given as $\left\langle \phi, \frac{q}{C_q} \Gamma_{\theta_x,q} \right\rangle_{OKHS} = q \int_0^1 (1-\tau)^{q-1} \phi[\theta_x](\tau) d\tau.$ Through this assignment, $H_{OKHS}$ yields a RKHS with kernel $K(x,y) = \langle \Gamma_{\theta_y,q},\Gamma_{\theta_x,q} \rangle_{OKHS}.$
% \end{theorem}

% \begin{proof}

% \end{proof}
% }
\section{Nonlocal Liouville Operators}\label{sec:nonlocal-ops}

The objective of this manuscript is to generalize existing methods for integer order dynamical systems for system identification and DMD analysis to that of fractional order dynamical systems of Caputo type. This section introduces two fractional order generalizations of the Liouville operator, which will be employed in the sequel for DMD analysis.

 \begin{definition}Let $\mathcal{D}(\mathcal{A}_{f,q}) := \{ \varphi \in H^{q}_{OKHS} : J^{1-q}\left(\nabla \varphi\cdot D^{1-q}f \right)\in H^q_{OKHS} \}$. The \emph{fractional} Liouville operator, $\mathcal{A}_{f,q}: \mathcal{D}(\mathcal{A}_{f,q}) \to H^q_{OKHS}$, corresponding to \eqref{eq:caputo-dynamics} over $H^q_{OKHS}$ is given as $\mathcal{A}_{f,q} \varphi[\gamma](t) = \frac{1}{\Gamma(q)} \int_0^{t} (t - \tau)^{-q} \nabla \varphi(\gamma(\tau)) \cdot D^{1-q} f(\gamma(\tau)) d\tau$.\end{definition}
% }

% Let $D_*^q x(t) = f(x(t))$ be a dynamical system of Caputo type with $0 < q < 1$, and recall that \[ \dot x(t) = D^{1-q} f(x(t)). \] Let $H_{OKHS}$ be an occupation kernel Hilbert space of order $1$ and $\tilde H_{OKHS}$ an occupation kernel Hilbert space of order $q$, each arising from a RKHS, $H_{RKHS}$, through the canonical identification as given in Section \ref{sec:okhs}. For a function $\varphi \in H_{OKHS}$, the gradient of $\varphi$ is given through its canonical identification, $\varphi = \phi_g$ for $g \in H_{RKHS}$, as $\nabla \varphi := \phi_{\nabla g}$.

% \begin{definition}\label{def:liouville}Let $\mathcal{D}(\mathcal{A}_{f,q}) := \{ \varphi \in H_{OKHS} : \nabla \varphi\cdot D^{1-q}f \in H_{OKHS} \}$. The Liouville operator, $\mathcal{A}_{f,q} : \mathcal{D}(\mathcal{A}_{f,q}) \to H_{OKHS}$, corresponding to the dynamical system in \eqref{eq:caputo-dynamics} over an $H_{OKHS}$ arising from a RKHS is defined as $\mathcal{A}_{f,q} \varphi[\gamma](t) = \nabla g(\gamma(t)) D^{1-q} f(\gamma(t))$.

% Alternatively, let $\mathcal{D}(\mathcal{A}_{f,q}) := \{ \varphi \in H_{OKHS} : J^{1-q} \nabla \varphi\cdot D^{1-q}f \in H_{OKHS} \}$. The \emph{fractional} Liouville operator, $\mathcal{A}_{f,q}: \mathcal{D}(\mathcal{A}_{f,q}) \to H_{OKHS}$, corresponding to \eqref{eq:caputo-dynamics} over $\tilde H_{OKHS}$ is given as $\mathcal{A}_{f,q} \varphi[\gamma](t) = \frac{1}{\Gamma(q)} \int_0^{t} (t - \tau)^{-q} \nabla \varphi(\gamma(\tau)) D^{1-q} f(\gamma(\tau)) d\tau$.\end{definition}

%Got to here with recent edits December 27 2020

The significance of the definition of this fractional Liouville operator manifests through the examination of the performance of its eigenfunctions on a trajectory, $\gamma$, satisfying \eqref{eq:caputo-dynamics}. 
% Specifically, suppose that $\phi_{g} \in H_{OKHS}$ is an eigenfunction of $\mathcal{A}_{f,q}$ with eigenvalue $\lambda$, then
% \begin{equation}\label{eq:eigenfunction-relation-liouville}
% \frac{d}{dt} \phi_{g}[\gamma](t) = \nabla g(\gamma(t)) \dot \gamma(t) = \nabla g( \gamma(t) ) D^{1-q} f(\gamma(t)) = \mathcal{A}_{f,q} \phi_{g}[\gamma](t) = \lambda \phi_{g}[\gamma](t),    
% \end{equation}
% hence $\phi_{g}[\gamma](t) = \phi_g[\gamma](0) e^{\lambda t}.$ This relation on the eigenfunctions of the Liouville operator agrees with that expressed for continuous time integer order systems in \citep{rosenfeldDMD}.
Suppose that $\phi_{h} \in \tilde H_{OKHS}$ is an eigenfunction for ${\mathcal{A}}_{f,q}$ with eigenvalue $\lambda$, then
\begin{gather}\label{eq:eigenfunction-relation-fracliouville} D_*^q \phi_{h}[\gamma](t) = C_q \int_0^{t} (t - \tau)^{-q} \nabla h(\gamma(\tau)) \dot \gamma(\tau) d\tau\\
= C_q \int_0^{t} (t - \tau)^{-q} \nabla h(\gamma(\tau)) D^{1-q} f(\gamma(\tau)) d\tau = {\mathcal{A}}_{f,q}\phi_{h}[\gamma](t) = \lambda \phi_{h}[\gamma](t).\nonumber \end{gather}
Hence, $\phi_{h}[\gamma](t) = \phi_{h}[\gamma](0) E_q(\lambda t^q),$ where $E_q$ is the Mittag-Leffler function of order $q$, given as $E_q(t) := \sum_{m=0}^\infty \frac{t^m}{\Gamma(qm+1)}.$

The decompositions of this operator will leverages Mittag-Leffler functions. Specifically, the decomposition of the dynamics given through the fractional Liouville operator, $\mathcal{A}_{f,q}$, will be a linear combination of Mittag-Leffler functions, which are likely to yield better predictions for fractional order dynamical systems, given the intimate relationship between fractional derivatives and Mittag-Leffler functions.

% Interestingly, the best choice of RKHS for the foundation of an OKHS depends on the selection of the approach to the generalization of Liouville operators. In particular, the exponential dot product kernel's native space (frequently referred to as the Bargman-Fock space), aligns well with the first type of Liouville operator, $\mathcal{A}_{f,q}$, since its kernel functions are of the form $K(x,y) = \exp(x^Ty)$. However, for the fractional Liouville operator, $\mathcal{A}_{f,q}$, a multi-variable generalization of the Mittag-Leffler RKHS of the slitted plane (cf. \citep{rosenfeld2018mittag}) would be more appropriate, where $K(s,t) = E_q(s^q t^q)$.

Similar to the case of the integer order Liouville operator, the fractional Liouville operator interacts nicely with occupation kernels from their respective spaces. In particular, 
% \begin{equation}\label{eq:liouville-occ-relation}\langle \mathcal{A}_{f,q} \phi_{g}, \Gamma_{\gamma} \rangle_{H_{OKHS}} = g(\gamma(T)) - g(\gamma(0)) = \phi_g[\gamma](T) - \phi_g[\gamma](0), \end{equation}
% and
\begin{equation}\label{eq:fracliouville-occ-relation}\langle \mathcal{A}_{f,q} \phi_h, \Gamma_{\gamma} \rangle_{\tilde H_{OKHS}} = \phi_h[\gamma](T) - \phi_h[\gamma](0), \end{equation}
where $\Gamma_{\gamma}$ is the occupation kernel corresponding to $\tilde H_{OKHS}.$ At this point, it is important to note that each OKHS has ``bounded point evaluations'' through constant trajectories, $\alpha_{x_0} : [0,1] \to \mathbb{R}^n$ where $\alpha_{x_0}(t) = x_0$. Point evaluation is then expressed as $\langle \phi_g, \Gamma_{\alpha_{x_0}} \rangle_{\tilde H_{OKHS}} = C_q \int_0^1 (1-t)^{q-1} \phi_g[\alpha_{x_0}](t) dt = C_q \int_0^1 (1-t)^{q-1} g(x_0) dt = \frac{C_q}{q} g(x_0) = \frac{C_q}{q} \phi_g[\alpha_{x_0}](0).$ Combining this result with the inner product relations yields
\begin{equation}\mathcal{A}_{f,q}^* \Gamma_{\gamma} = C_q \left(\Gamma_{\alpha_{\gamma(T)}} - \Gamma_{\alpha_{\gamma(0)}}\right).\end{equation}
These relations will be pivotal in the development of a DMD method for fractional order systems.

\section{Finite Rank Representations of Liouville Operators}\label{sec:finite-rank}

For $q > 0$, the objective of this section is to leverage observed trajectories, $\mathcal{M} = \{ \gamma_{i} : [0,T_i] \to \mathbb{R}^n \}_{i=1}^M$, that satisfy $D_*^q \gamma_i = f(\gamma_i)$ for each $i=1,\ldots,M$ and an unknown $f : \mathbb{R}^n \to \mathbb{R}^n$ to derive a model, $G:[0,T]\to\mathbb{R}^n$, for the system for which given an initial value $x(0) \in \mathbb{R}^n$ the trajectory $x:[0,T]\to\mathbb{R}^n$ satisfying $D_*^q x = f(x)$ can be estimated as $x(t) \approx G(t).$ To obtain this model, a finite rank representation of the operator $\mathcal{A}_{f,q}$ over a given OKHS, $H$, will be determined through the associated occupation kernels and the relations established in Section \ref{sec:nonlocal-ops}. The finite rank representation of the operator will then be leveraged as a proxy for the actual densely defined operator, where a spectral decomposition will be determined and the eigenfunctions will form the foundation of the data driven model. In practice, the eigenfunctions will be determined as a linear combination of the associated occupation kernels. Note that the following construction assumes that the occupation kernels are in the domain of the fractional Liouville operator.

\subsection{Finite Rank Representation of the Liouville Operator,  $\mathcal{A}_{f,q}$}

Let $H$ be an OKHS of order $q$ determined through a RKHS, $\tilde H$, of continuously differentiable functions over $\mathbb{R}^n$ with kernel $K$. Let $\beta$ be the ordered basis of occupation kernels corresponding to $\mathcal{M}$ given as $\beta := \{ \Gamma_{\gamma_i,q} \}_{i=1}^M$. The vector space $\vspan \beta$ is a finite dimensional subspace of $H$ and hense, closed. Let $P_{\beta}$ denote the projection operator from $H$ onto $\vspan \beta$. As demonstrated in Section \ref{sec:nonlocal-ops}, each occupation kernel corresponding to $\gamma_i$ is in the domain of the adjoint of the operator $\mathcal{A}_{f,q}$. Hence, the operator $P_{\beta} \mathcal{A}_{f,q} P_{\beta}$ is well defined over $H$, and this operator is of finite rank. Note that the matrix, $[P_{\beta} \mathcal{A}_{f,q} P_{\beta}]_{\beta}^\beta$, defined in terms of the ordered basis $\beta$ may be expressed simply as $[P_{\beta} \mathcal{A}_{f,q} ]_{\beta}^\beta$, since the domain of definition for the matrix is $\vspan \beta$ and $P_{\beta} \phi_h = \phi_h$ for all $\phi_h \in \vspan \beta$.

For a function $\phi_h \in H$, the projection of $H$ onto $\vspan \beta$ may be written as $P_{\beta} \phi_h = \sum_{i=1}^M w_i \Gamma_{\gamma_i,q}$, where $w = (w_1,\ldots, w_m)^T \in \mathbb{R}^M$ is obtained by solving the matrix equation
\begin{equation} \label{eq:projection_equation}
    \begin{pmatrix}
        \langle \Gamma_{\gamma_1,q}, \Gamma_{\gamma_1,q} \rangle_H & \cdots & \langle \Gamma_{\gamma_M,q}, \Gamma_{\gamma_1,q} \rangle_H\\
        \vdots & \ddots & \vdots\\
        \langle \Gamma_{\gamma_1,q}, \Gamma_{\gamma_M,q} \rangle_H & \cdots & \langle \Gamma_{\gamma_M,q}, \Gamma_{\gamma_M,q} \rangle_H
    \end{pmatrix}
    \begin{pmatrix}
        w_1 \\ \vdots \\ w_M
    \end{pmatrix}
    =
    \begin{pmatrix}
        \langle \phi_h, \Gamma_{\gamma_1,q} \rangle_H\\
        \vdots \\
        \langle \phi_h, \Gamma_{\gamma_M,q} \rangle_H
    \end{pmatrix}.
\end{equation}

% Consequently, the finite rank representation of $\mathcal{A}_{f,q}^*$ may be determined as 
% \begin{gather} \label{eq:finite-rank-liouille}
%     [P_{\beta}\mathcal{A}_{f,q}^*]_{\beta}^\beta = \begin{pmatrix}
%         \langle \Gamma_{\gamma_1,q}, \Gamma_{\gamma_1,q} \rangle_H & \cdots & \langle \Gamma_{\gamma_M,q}, \Gamma_{\gamma_1,q} \rangle_H\\
%         \vdots & \ddots & \vdots\\
%         \langle \Gamma_{\gamma_1,q}, \Gamma_{\gamma_M,q} \rangle_H & \cdots & \langle \Gamma_{\gamma_M,q}, \Gamma_{\gamma_M,q} \rangle_H
%     \end{pmatrix}^{-1} \times\\
%     \begin{pmatrix}
%          \langle \mathcal{A}_{f,q}^* \Gamma_{\gamma_1,q}, \Gamma_{\gamma_1,q} \rangle_H & \ldots & \langle \mathcal{A}_{f,q}^* \Gamma_{\gamma_M,q},\Gamma_{\gamma_1,q} \rangle_H\\
%         \vdots \\
%         \langle \mathcal{A}_{f,q}^* \Gamma_{\gamma_1,q}, \Gamma_{\gamma_M,q} \rangle_H & \ldots & \langle \mathcal{A}_{f,q}^* \Gamma_{\gamma_M,q},\Gamma_{\gamma_M,q} \rangle_H
%     \end{pmatrix}.\nonumber
% \end{gather}
% In \eqref{eq:finite-rank-liouille}, the entries of the matrix may be computed using \ref{eq:liouville-occ-relation} and the canonical identification $\phi_g$ developed in Section \ref{sec:okhs}. To wit, noting that $q = 1$ for the Liouville operator and occupation kernel relation:
% $\langle \mathcal{A}_{f,q}^* \Gamma_{\gamma_i,q}, \Gamma_{\gamma_j,1} \rangle_H = \langle \Gamma_{\alpha_{\gamma_i(T)},1} - \Gamma_{\alpha_{\gamma_i(0)},1}, \Gamma_{\gamma_{j},q} \rangle_H
% = \int_{0}^T  \tilde K(\gamma_{j}(\tau),\gamma_i(T)) - \tilde K(\gamma_{j}(\tau),\gamma_i(0)) d\tau.$

% Similarly, 
For the fractional Liouville operator, $\mathcal{A}_{f,q}$, the finite rank representation is given as
\begin{gather} \label{eq:finite-rank-fracliouille}
    [P_{\beta}\mathcal{A}_{f,q}^*]_{\beta}^\beta = \begin{pmatrix}
        \langle \Gamma_{\gamma_1,q}, \Gamma_{\gamma_1,q} \rangle_H & \cdots & \langle \Gamma_{\gamma_M,q}, \Gamma_{\gamma_1,q} \rangle_H\\
        \vdots & \ddots & \vdots\\
        \langle \Gamma_{\gamma_1,q}, \Gamma_{\gamma_M,q} \rangle_H & \cdots & \langle \Gamma_{\gamma_M,q}, \Gamma_{\gamma_M,q} \rangle_H
    \end{pmatrix}^{-1} \times\\
    \begin{pmatrix}
         \langle \mathcal{A}_{f,q} \Gamma_{\gamma_1,q}, \Gamma_{\gamma_1,q} \rangle_H & \ldots & \langle \mathcal{A}_{f,q} \Gamma_{\gamma_M,q},\Gamma_{\gamma_1,q} \rangle_H\\
        \vdots \\
        \langle \mathcal{A}_{f,q} \Gamma_{\gamma_1,q}, \Gamma_{\gamma_M,q} \rangle_H & \ldots & \langle \mathcal{A}_{f,q} \Gamma_{\gamma_M,q},\Gamma_{\gamma_M,q} \rangle_H
    \end{pmatrix},\nonumber
\end{gather}
where each entry can be computed as
\begin{gather*}
\langle \mathcal{A}_{f,q} \Gamma_{\gamma_i,q}, \Gamma_{\gamma_j,q} \rangle_H = \langle  \Gamma_{\gamma_i,q}, \mathcal{A}_{f,q}^* \Gamma_{\gamma_j,q} \rangle_H = \langle \Gamma_{\gamma_i,q}, \Gamma_{\alpha_{\gamma_j(T)},q} - \Gamma_{\alpha_{\gamma_j(0)},q} \rangle_H\\
= C_q \int_{0}^T (T-\tau)^{q-1} \left( K(\gamma_{i}(\tau),\gamma_j(T)) -  K(\gamma_{i}(\tau),\gamma_j(0))\right) d\tau.
\end{gather*}

The entries of the Gram matrix in \eqref{eq:finite-rank-fracliouille} can be expressed as 
\begin{gather*}
\langle \Gamma_{\gamma_i,q}, \Gamma_{\gamma_j,q} \rangle_H
= (C_q)^2 \int_{0}^T (T-\tau)^{q-1}(T-t)^{q-1} K(\gamma_{j}(\tau),\gamma_i(t)) d\tau dt,
\end{gather*}
with $q$ matching the order of the system for the fractional Liouville operator.

\section{Dynamic Mode Decompositions with Liouville Modes and Fractional Liouville Modes}

In Dynamic Mode Decomposition (DMD) methods for the data driven analysis of dynamical systems, the identity function (also called the \emph{full state observable}), $g_{id} : \mathbb{R}^n \to \mathbb{R}^n$, given as $g_{id}(x) = x$ is individually decomposed with respect to an eigenbasis in a RKHS corresponding to finite rank representations of Koopman and Liouville operators similar to those determined in Section \ref{sec:finite-rank}. The result is a linear combination of scalar valued eigenfunctions multiplied by a collection of vectors obtained through the projection of the component of $g_{id}$ onto the eigenbasis. These vectors are called the \emph{dynamic modes} of the system. When connected with particular operators, they are also called Koopman or Liouville modes. Subsequently, a model for the dynamical system is determined by exploiting certain features of the eigenfunctions, which ultimately replace the eigenfunctions with exponential functions.

This section discusses the two different models that are determined through a choice of using the Liouville operator or the fractional Liouville operator over a OKHS. In place of the identity function, whose image is in $\mathbb{R}^n$, the DMD method presented here leveraged the signal valued analogue, $\phi_{g_{id}}$, and exploits identities \eqref{eq:finite-rank-fracliouille}.

From the data driven perspective, the Liouville operator and Fractional Liouville operator for a particular collection of trajectories is not directly accessible. This motivates the use of finite rank representations such as those given in Section \ref{sec:finite-rank}. An eigenvector, $v = (v_1,\ldots,v_M)^T \in \mathbb{C}^M$, with eigenvalue $\lambda \in \mathbb{C}$ obtained from \eqref{eq:finite-rank-fracliouille} corresponds to a function in the OKHS as $\varphi = \frac{1}{\sqrt{v^T G v}} \sum_{i=1}^M v_i \Gamma_{\gamma_i,q}$ with $q$ as the fractional order of the system in \eqref{eq:finite-rank-fracliouille}, which is an eigenfunction for the corresponding finite rank operator. To derive a model for the dynamical system, the eigenfunctions obtained from the finite rank operators will be used in place of proper eigenfunctions of the fractional Liouville operator. 

Specifically, suppose that the eigenfunctions $\{ \varphi_{i} \}_{i=1}^M$ corresponding to the eigenvalues $\{ \lambda_i \}$ diagonalize the finite rank operator $P_{\beta} \mathcal{A}_{f,q} P_{\beta}.$ The full state observable for the OKHS, $\phi_{g_{id}}$, may then be approximated through a projection as $\phi_{g_{id}} \approx \sum_{i=1}^M \xi_i \varphi_i$, where
\[ (\xi_1 \cdots \xi_M)^T := \left( W^{-1} \begin{pmatrix} \langle (\phi_{g_{id}})_{j},\varphi_1 \rangle_H & \cdots & \langle (\phi_{g_{id}})_j,\varphi_M \rangle_H\end{pmatrix}^T \right)_{j=1}^n \in \mathbb{C}^{m \times n}\]
are the fractional Liouville modes, and $W = ( \langle \varphi_i,\varphi_j \rangle_H )_{i,j=1}^M.$ Note that the inner products are taken one dimension at a time, and this is because the image of the functions $\varphi_i$ are scalar valued signals. By taking the inner product one dimension at a time. This is equivalent to working over a vector valued version of OKHSs, but without the added theoretical overhead. This decomposition of the full state observable allows the expression of an approximate model for the state, $D_{*}^q x(t) = f(x(t))$, by leveraging \eqref{eq:eigenfunction-relation-fracliouville} as
%\[ x(t) = \phi_{g_{id}}[x](t) \approx \sum_{i=1}^M \xi_i \varphi_i[x](t) \approx \sum_{i=1}^M \xi_i \varphi_i[x](0) e^{\lambda_i t}. \]
%Hence, a model for the dynamical system has been derived from the observed trajectories. However, since the dynamics are governed by the Caputo fractional derivative, a model obtained from the finite rank representation of the fractional Liouville operator may be more suitable. In this case, suppose that $\{ \hat \varphi_{i} \}_{i=1}^M$ are eigenfunctions of the finite rank operator $\mathcal{A}_{f,q}$ with eigenvalues $\{ \lambda_i \}_{i=1}^M$. Following the same procedure as above and
Using \eqref{eq:eigenfunction-relation-fracliouville}, the resultant model for the dynamical system is given as
\[ x(t) = \phi_{g_{id}}[x](t) \approx \sum_{i=1}^M  \xi_i \varphi_i[x](t) \approx \sum_{i=1}^M  \xi_i \varphi_i[x](0) E_q(\lambda_i t^q), \]
with $\xi_i$ being the $i$-th fractional Liouville mode. Hence, a model for the dynamical system has been derived from the observed trajectories. %Heuristically, the latter model is expected to match the state more closely, since $t \mapsto E_q(\lambda_i t^q)$ is the solution to a fractional order linear system. Additionally, if the finite sums above are replaced with a series obtained from a collection of eigenfunctions that diagonalize the original operators, then the expressions are expected to result in equalities.
% \section{Computational Considerations}

% \subsection{Occupation Kernels and Gram Matrices}

% \section{Comparison of the Dynamic Modes}

\section{Discussion}

%It should be noted that while the theoretical developments given above are for a scalar valued trajectories, these methods extend naturally to vector valued quantities by treating each dimension separately and then later stacking the dynamic modes. This is equivalent to using vector valued kernels that are diagonal as the fundamental space from which the OKHS is constructed over.

Despite the theoretical developments in the manuscript, the actual implementation of the DMD method does not differ greatly from standard occupation kernel DMD. As seen in Section \ref{sec:finite-rank}, the computations ultimately are performed using the kernel of the underlying RKHS rather than directly on the OKHS. Hence, the computation of Gram matrices only need a slight adjustment to accommodate the fractional integrals. In fact, these computations are exactly in line with \citep{li2020fractional}, and follow from Newton-Cotes and Gaussian quadrature methods.

Finally, the fractional order framework can also be posed over signal valued RKHSs \citep{rosenfeld2021theoretical}, which gives another approach to handling nonlocal data interpretation. OKHSs differ in that the definition of OKHSs are independent of that of RKHSs, where the RKHSs are used in this manuscript as a particular realization of OKHSs. Moreover, while signal valued spaces correspond to trajectories of a fixed length (and may artificially adjust the lengths of trajectories using indicator functions), OKHSs allow for arbitrarily long trajectories natively.

One drawback of this method is in the reconstruction. It was observed by the authors that the reconstruction of time series data that was integer order was obtainable for short timesteps, the major limitation was in the evaluation of the Mittag-Leffler functions themselves, where efficient routines for high accuracy computations do not match the computational efficiency of the exponential functions. This is a major drawback, which will be ameliorated as the fractional order community continues on improving numerical methods for the evaluation of Mittag-Leffler functions.

\section{Conclusion}

This manuscript presented the Hilbert space framework of OKHSs for the analysis of nonlocal operators, such as the fractional Liouville operator. OKHSs are a natural extension of RKHSs, and it was shown that an OKHS may be obtained from a suitable RKHS. As an example of the flexibility of this framework, a DMD routine was developed for a Caputo time-fractional dynamical system, where the resultant algorithm is remarkably close to that of occupation kernel DMD methods for integer order systems.

\begin{ack}
This research was supported by the Air Force Office of Scientific Research (AFOSR) under contract numbers FA9550-20-1-0127 and FA9550-21-1-0134, and the National Science Foundation (NSF) under award 2027976. Any opinions, findings and conclusions or recommendations expressed in this material are those of the author(s) and do not necessarily reflect the views of the sponsoring agencies.
\end{ack}

% \section{Broader Impacts Statement}

% This work gives a Hilbert space framework for where non-local operators may be well defined for spectral analysis. This framework elects to use trajectory information as the fundamental unit of information, rather than local point evaluations as is the case for RKHSs. This work opens the door for the analysis of a variety of fractional order dynamical systems through small tweaks to the occupation kernel definitions. This work is not expected to have any negative societal impacts.
% \appendix

% \section{Appendix}

% Optionally include extra information (complete proofs, additional experiments and plots) in the appendix.
% This section will often be part of the supplemental material.

\bibliographystyle{plainnat}
\bibliography{references,extrarefs}

\end{document}